\documentclass{amsart}
\RequirePackage{amssymb, amsfonts, amsmath, latexsym, verbatim, xspace, setspace, mathrsfs,graphicx,bm,multirow}

\newcommand{\R}{\mathbb{R}}

\newcommand{\Z}{\mathbb{Z}}
\newcommand{\Q}{\mathbb{Q}}
\newcommand{\C}{\mathbb{C}}

\newcommand{\lb}{\left\{}
\newcommand{\rb}{\right\}}
\newcommand{\df}{\displaystyle\frac}

\newtheorem{thm}{Theorem}[section]
\newtheorem{prop}[thm]{Proposition}
\newtheorem{lem}[thm]{Lemma}

\theoremstyle{definition}
\newtheorem{definition}[thm]{Definition}
\newtheorem{example}[thm]{Example}
\newtheorem{question}[thm]{Question}

\title[Salem number stretch factors]{Salem number stretch factors and totally real fields arising from Thurston's construction}
\date{\today}
\begin{document}
\author[J. ~Pankau]{Joshua Pankau}
\address{Department of Mathematics, University of California,Santa Barbara}
\email{jpankau@math.ucsb.edu}

\begin{abstract} 
In 1974, Thurston proved that, up to isotopy, every automorphism of closed orientable surface is either periodic, reducible, or pseudo-Anosov. The latter case has lead to a rich theory with applications ranging from dynamical systems to low dimensional topology. Associated with every pseudo-Anosov map is a real number $\lambda > 1$ known as the \textit{stretch factor}. Thurston showed that every stretch factor is an algebraic unit but it is unknown exactly which units can appear as stretch factors. In this paper we show that every Salem number has a power that is the stretch factor of a pseudo-Anosov map arising from a construction due to Thurston. We also show that every totally real number field $K$ is of the form $K = \Q(\lambda + \lambda^{-1})$, where $\lambda$ is the stretch factor of a pseudo-Anosov map arising from Thurston's construction.
\end{abstract}

\maketitle
\section{Introduction}

A homeomorphism $\phi$ from a closed orientable surface $S_g$ to itself is called \textit{pseudo-Anosov} if there is a pair of transverse measured folations $\mathcal{F}_s$ and $\mathcal{F}_u$ of $S_g$ in which $\phi$ stretches $\mathcal{F}_u$ by a real number $\lambda > 1$, and contracts $\mathcal{F}_s$ by a factor of $\lambda^{-1}$. The number $\lambda$ is known as the \textit{stretch factor} of $\phi$. Thurston showed in \cite{MR956596} that the stretch factor of any pseudo-Anosov map is an algebraic unit whose degree over $\Q$ is bounded above by $6g-6$. It is an open problem as to exactly which algebraic units appear as stretch factors of pseudo-Anosov maps.

\vspace{1em}
There are several well known constructions of pseudo-Anosov maps. Thurston provided a construction, known as \textit{Thurston's construction}, which describes pseudo-Anosov maps as products of Dehn twists around simple closed curve that divide the surface into disks. We will give an overview of this construction in the next section but for a more complete treatment see either \cite{FLP}, Expos\'e 13, or  \cite{MR2850125}. In \cite{MR930079}, Penner describes another construction involving products of Dehn twists. These constructions can produce the same pseudo-Anosov map but there are pseudo-Anosov maps that arise from one and not the other. In a recent paper \cite{MR3447112}, Shin and Strenner showed that if $\lambda$ is a stretch factor of a pseudo-Anosov map coming from Penner's construction then $\lambda$ does not have Galois conjugates on the unit circle. Whereas, Veech showed that if $\lambda$ is a stretch factor coming from Thurston's construction then $\Q(\lambda + \lambda^{-1})$ is a totally real number field, for which we provide a proof in the next section.

\vspace{1em}
In this paper we will focus on Thurston's construction, and a certain type of algebraic unit known as a \textit{Salem number}. Salem numbers have complex conjugates on the unit circle so they cannot arise as stretch factors from Penner's construction. On the other hand, there are Salem numbers that appear as stretch factors from Thurston's construction, and many of the smallest known stretch factors are Salem numbers. A natural question to ask is if we can get every Salem number as a stretch factor. We prove the following:

\vspace{1em}
\textbf{Theorem A.} \textit{For any Salem number $\lambda$, there are positive integers $k$ and $g$ such that $\lambda^k$ is the stretch factor of a pseudo-Anosov homeomorphism of $S_g$ arising from Thurston's construction, where $g$ depends only on the degree of $\lambda$ over $\Q$.}

\vspace{1em}
Since we know that $\Q(\lambda + \lambda^{-1})$ is a totally real number field when $\lambda$ is a stretch factor from Thurston's construction, it is natural to ask which totally real number fields arise this way. We also prove:

\vspace{1em}
\textbf{Theorem B.} \textit{Every totally real number field is of the form $K = \Q(\lambda + \lambda^{-1})$, where $\lambda$ is the stretch factor of a pseudo-Anosov map arising from Thurston's construction.}

\vspace{1em}
The structure of the paper is as follows: In section 2 we discuss Thurston's constructions, highlighting the information that will be relevant to results of the paper. In section 3 we discuss a new way of constructing a closed orientable surface from a nonsingular positive integer matrix. This will allow us to convert our algebraic results back into topological information, which will be important for proving both Theorem A and B. In section 4 we define Salem numbers and discuss their relation to Thurston's construction, then prove Theorem A in section 5. In sections 6 and 7 we develop some algebraic number theory and conclude by proving Theorem B.

\vspace{1em}
\textbf{Acknowledgements.} I would like to thank my advisor Darren Long for many helpful conversations and advice given as these ideas took shape. I would also like to thank Michael Dougherty and David Wen for listening to me talk at length about these results. Last, but not least, I would like to thank my wife Jacqueline for being my emotional support and for teaching me the basics of Adobe Illustrator so that I could make the graphics.

\section{Thurston's Construction}

In this section we give an overview of Thurston's construction of pseudo-Anosov maps. Our goal is to provide the basic framework of the construction and establish some of the key ideas that we will use in this paper. For a more in depth discussion see \cite{FLP}. We also provide a proof of the result from Veech that if $\lambda$ is a stretch factor of a pseudo-Anosov map arising from this construction then $\Q(\lambda + \lambda^{-1})$ is a totally real number field. We begin with the following definitions:

\begin{definition}
A \textit{multicurve} $C = \lb C_1, C_2,...,C_n\rb$ on a surface $S$ is a union of finitely many disjoint simple closed curves on $S$. For the purposes of this paper we will also require that each $C_i$ is \textit{essential}, i.e., not isotopic to a point, and no two $C_i$'s are parallel, i.e., not isotopic.
\end{definition}

\begin{definition}
Let $C = \lb C_1, ..., C_n \rb$ and $D = \lb D_1,...,D_m\rb$ be a pair of multicurves on a closed orientable surface $S_g$. We say that $C \cup D$ \textit{fills} $S_g$ if every component of $S_g - (C \cup D)$ is a disk. We also say that $C$ and $D$ are \textit{tight} if $i(C_i,D_j) = |C_i \cap D_j|$, where $i(\cdot,\cdot)$ denotes the geometric intersection number.
\end{definition}

We will denote the Dehn twist about a simple closed curve $\alpha$ by $T_{\alpha}$, with the convention that we are twisting to the right with respect to the orientation of $S_g$. We now give an overview of Thurston's construction following the discussion given in \cite{Long}.

\begin{thm}[Thurston's Construction] Suppose that $C = \lb C_1,...,C_n \rb$, $D = \lb D_1, ..., D_m\rb$ are tight, filling multicurves on a closed, orientable surface $S_g$. To each $C_i \in C$ assign an integer $n_i > 0$ and to each $D_j \in D$ assign an integer $m_j > 0$. Then the maps
	\[
		\begin{array}{lcr}
				T_C = \displaystyle\prod_{i}T_{{C}_i}^{n_i} \hspace{.2in} & \textit{and} & \hspace{.2in} T_D = \displaystyle\prod_{j}T_{{D}_j}^{m_j}
			\end{array}
	\]
can be represented by matrices in PSL$(2;\R)$ and any word $\theta = w(T_C,T_D)$ which corresponds to a hyperbolic class $[\theta] \in \text{SL}(2;\R)$ is pseudo-Anosov with stretch factor the larger of the two eigenvalues.
\end{thm}

\vspace{1em}
The first step in this construction is to give a branched flat structure to $S_g$, that is, a way to view the surface as a flat manifold except at a finite number of points. We do this by describing a way of decomposing the surface into glued up rectangles. Let $C = \lb C_1,...,C_n\rb$ and $D = \lb D_1,...,D_m\rb$ be tight, filling multicurves on $S_g$. $C\cup D$ defines a cell structure on $S_g$ since the complement of $C \cup D$ is a union of disks. We define a dual cell structure of $C \cup D$ where we take a co-vertex for each cell, a co-edge for each edge, and a co-cell for each vertex. In a neighborhood of each vertex, we see four segments originating from the vertex so each co-cell will have four sides. We can think of the dual cell structure as placing rectangles on each vertex and then identify sides of the rectangles that have an arc of some $C_i$ or $D_j$ between them.

\begin{figure}[ht]
		\centering
	  \includegraphics[scale=.3]{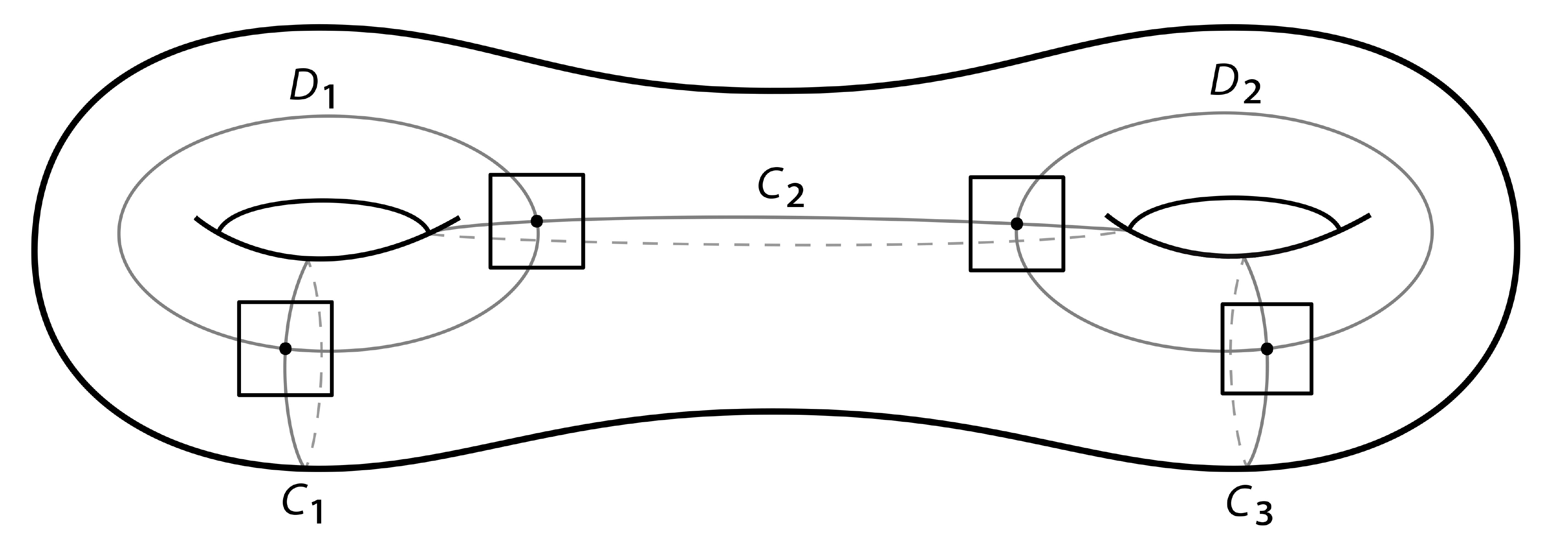}
		\caption{A pair of tight, filling multicurves on $S_2$ with co-cells at each vertex.}
	\end{figure}

Assigning lengths to sides of the rectangles allows us to identify each co-cell with a rectangle in $\R^2$. So we can view $S_g$ as a union of these metric rectangles, hence we have given $S_g$ a branched flat structure.

\begin{figure}[ht]
		\centering
	  \includegraphics[scale=.3]{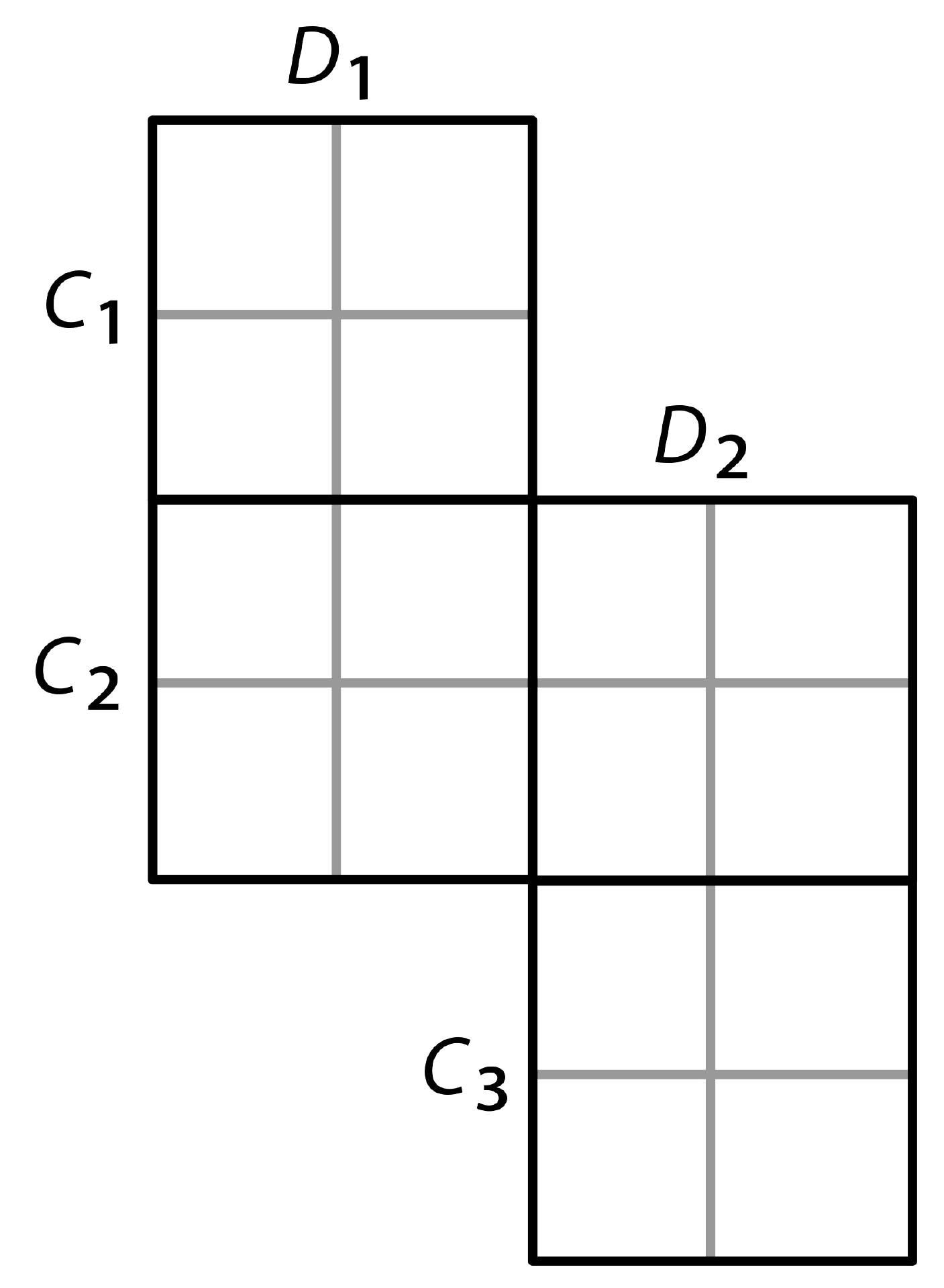}
		\caption{$S_2$ as a union of rectangles.}
	\end{figure}
	
\newpage
	
There is quite a bit of choice for the length of the sides of each rectangle and this freedom can be used to show the following:

\begin{prop}Assign a positive integer $n_i$ to each $C_i \in C$ and a positive integer $m_j$ to $D_j$. Then there is a branched flat structure on $S_g$ for which both the maps $$T_C = \displaystyle\prod_{i}T_{C_i}^{n_i} \hspace{.1in} \textit{and} \hspace{.1in} T_D = \displaystyle\prod_{j}T_{D_j}^{m_j}$$ are affine.
\end{prop}

\vspace{1em}
By \textit{affine} we mean that lines in the branched flat structure are mapped to lines. The Dehn twists about the $C_i$ and $D_j$ act on the branched flat structure by skewing the rectangles, so we wish to find a choice for side lengths so that the slope of the sides after twisting $n_i$ times about $C_i$ is constant in $i$. Similarly, twisting $m_j$ times about $D_j$ is constant in $j$. We require that the height of each rectangle lying along $C_i$ to have constant height $h_i$ and that each rectangle lying along $D_j$ has constant width $\ell_j$. If we let $i(D_r,C_s) = Q_{rs}$ then it can be shown that $T_C$ is affine if
	\[\tan(\theta) = \df{h_i}{n_i\displaystyle\sum_{r}\ell_rQ_{ri}} \hspace{.2in} \text{for all} \ i,\]
where $\theta$ is the angle between the image of the vertical edges along $C_i$ with the horizontal edges along $C_i$ after skewing by $T_C$. Similarly, the condition that $T_D$ is affine is 
	\[\tan(\phi) = \df{\ell_j}{m_j\displaystyle\sum_{s}h_sQ_{sj}} \hspace{.2in} \text{for all} \ j.\]

It can be shown that the choice of $\ell_j$ are from a Perron-Frobenius eigenvector with eigenvalue $\nu$ of the matrix $MQNQ^T$ where $Q_{ij} = i(D_i,C_j)$ and $M$ and $N$ are diagonal matrices whose diagonal entries are $m_j$ and $n_i$, respectively. The $h_i$ come from a Perron-Frobenius eigenvector, also with eigenvalue $\nu$, of the matrix $QNQ^TM$. These give dimensions of the rectangles in which $T_C$ and $T_D$ are affine. With our chosen convention of twisting, these lengths give the following representations (after some rescaling of the $\ell_i$):
\begin{align} 
	[T_C] = \begin{bmatrix} 1 & 1 \\ 0 & 1\end{bmatrix} \hspace{.2in} \text{and} \hspace{.2in} [T_D] = \begin{bmatrix} 1 & 0 \\ -\nu & 1\end{bmatrix}
	\end{align}
which are matrices in $PSL(2;\R)$. Now any word $\phi = w(T_C,T_D)$ such that $[\phi] \in SL(2;\R)$ is hyperbolic ($|\text{tr($[\phi]$)}| > 2$) will be pseudo-Anosov. Since $[\phi]$ is hyperbolic it has two real eigenvalues, $\lambda$ and $\lambda^{-1}$, the stretch factor of $\phi$ is $|\lambda|$, and all lines parallel to the eigenspaces descend to transverse measured foliations on $S_g$. We end this section by proving the following:

\begin{thm}[Veech] If $\lambda$ is the stretch factor of a pseudo-Anosov map arising from Thurston's construction then $\Q(\lambda + \lambda^{-1})$ is a totally real number field.
\end{thm}

\begin{proof} Let $C$ and $D$ be tight, filling multicurves on a closed orientable surface, and let $M$, $N$ and $Q$ be nonnegative integer matrices defined as above.  Consider the matrix

	\[\Lambda = \begin{bmatrix}
									\mathbf{0} & M^{1/2}QN^{1/2}\\
									N^{1/2}Q^TM^{1/2} & \mathbf{0}
								\end{bmatrix}\]

$\Lambda$ is symmetric so it has real eigenvalues which means
	\[\Lambda^2 = \begin{bmatrix}
									M^{1/2}QNQ^TM^{1/2} & \mathbf{0}\\
									\mathbf{0}          & N^{1/2}Q^TMQN^{1/2}
								\end{bmatrix}\]	
has nonnegative eigenvalues, hence $M^{1/2}QNQ^TM^{1/2}$ has nonnegative eigenvalues. Conjugating this matrix by $M^{1/2}$ gives $MQNQ^T$, thus $MQNQ^T$ has nonnegative eigenvalues. This tells us that $\nu$ is totally nonnegative, so $\Q(\nu)$ is totally real. Now if $\lambda$ is a stretch factor arising from the above construction then $\lambda$ is a root of degree 2 polynomial in $\Q(\nu)[x]$ of the form
	\[ x^2 -(\lambda + \lambda^{-1})x + 1\]
so $\lambda + \lambda^{-1} \in \Q(\nu)$, and therefore $\Q(\lambda + \lambda^{-1})$ is also totally real. 
\end{proof}

\section{Constructing Surfaces from Positive Integer Matrices.}

In this section we describe a new construction in which we build a closed orientable surface $S$ from a nonsingular positive integer matrix $Q$, where on $S$ are two tight, filling multicurves $C$ and $D$ whose intersection matrix is $Q$. This construction serves as the central link between the algebraic results of this paper and topological ones. We will need it prove both Theorem A and Theorem B so we state and prove it here.

\begin{thm}
Given a nonsingular positive square integer matrix $Q$, there is a closed orientable surface $S$ with tight, filling multicurves $C =\lb C_1,...,C_n\rb$ and $D = \lb D_1,...,D_n \rb$ such that $Q_{ij} = i(D_i,C_j)$.
\end{thm}

\begin{proof} We start by taking $n$ rectangular strips where the $j$th strip is divided into $\displaystyle\sum_{i =1}^{n}Q_{ij}$ rectangles each oriented clockwise. Since the matrix $QQ^T$ is positive then it has a positive eigenvector with entries $\ell_i$. All vertical edges along the $j$th strip will have length $\ell_j$, where as $\ell_1$ will be the width of the first $Q_{1j}$ rectangles, and $\ell_2$ will be the width of the next $Q_{2j}$ rectangles, etc. We now take the central curve of each strip and call it $C_j$ for $j = 1,...,n.$

\vspace{1em}
Now we construct curves $D_i$ as follows: For a fixed $i \in \lb 1,...,n \rb$ there are $Q_{ij}$ rectangles of size $\ell_i \times \ell_j$ lying along the $j$th strip, and we imagine a curve that passes from the top of the left most $\ell_i \times \ell_j$ rectangle through the bottom, then wraps back around the $j$th strip and passes through the next $\ell_i \times \ell_j$ rectangle and continues this way until it reaches the right most $\ell_i \times \ell_j$ rectangle. The curve then continues to the $j+1$st strip and wraps around each $\ell_i \times \ell_{j+1}$ rectangle in similar fashion. After the curve wraps around each strip it closes up, so we have a simple closed curve which we call $D_i$. This gives us gluing instructions where we glue two edges, matching orientation, if they are connected by an arc of some $D_i$. We also identify the two vertical sides of each strip.
\begin{figure}[h]
		\centering
	  \includegraphics[scale=.25]{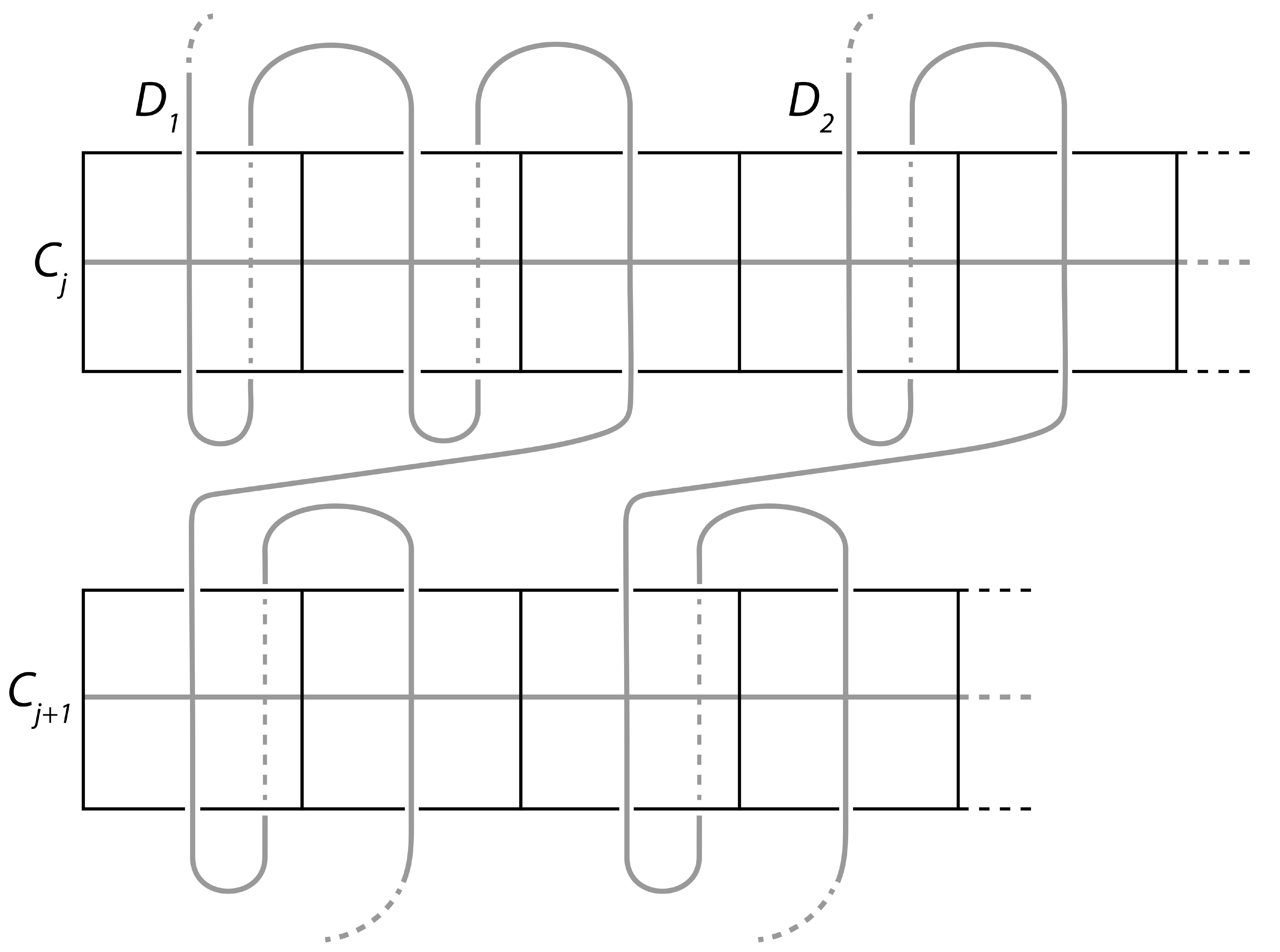}
		\caption{A piece of the cell structure. Edges of different cells are identified if there is an arc of some $D_i$ between them.}
	\end{figure}

Since we glue these strips together where each edge is identified with another, always matching orientation, then we get a closed orientable surface $S$. The condition that the matrix be nonsingular ensures that no two rows are the same, which means that none of the $D_i$ intersect a $C_j$ in exactly the same way, so no two components of $C$ or no two components of $D$ are parallel. By construction, every intersection of $D_i$ with $C_j$ has the same sign so the geometric and algebraic intersection numbers of  these curves are the same. This ensures that no $D_i$ forms a bigon with any $C_j$, hence the multicurves $C = \lb C_1,...,C_n\rb$ and $D = \lb D_1,...,D_n\rb$ are tight. Finally, taking the dual cell structure gives a cell structure of $S$ whose vertices are the points of intersection between each $D_i$ and $C_j$, and whose edges are arcs of some $D_i$ or $C_j$. Hence, the complement of $C \cup D$ are disks, so $C$ and $D$ fill $S$.
\end{proof}

\textbf{Remark.} We can do this construction more generally where you label each box along a $C_j$ strip with some ordering of each of the $D_i$'s, making sure that $D_i$ appears $Q_{ij}$-times. Then you make some choice of how to connect the various $D_i$ boxes across each strip, keeping in mind that the intersections must all have the same sign to ensure there are no bigons. For example, you can have $D_1$ wrap around $C_1$ twice then hit $C_2$ once, then $C_3$ twice, back up to $C_1$ once, etc. until all the $D_1$ boxes have been connected. This will still give an orientable surface with $C$ and $D$ as tight filling multicurves that intersect the correct number of times. For this more general construction, the genus of the surface is difficult to ascertain as different choices in arranging the curves can change the surface you obtain. For the construction given in the proof of Theorem 3.1 we can explicitly find the genus.

\begin{prop} 
Let $Q$ be a nonsingular positive integer $n\times n$ matrix whose entries are all greater than or equal to $2$. If $S$ is the closed orientable surface of genus $g$ obtained by following the construction given in the proof of Theorem $3.1$, then $g = n^2 - n + 1$.
\end{prop}

A proof of this is included in a forthcoming paper that will go into more detail about properties of this construction. For now we see that this gives an upper bound for the genus of the surface. If we us the general construction described in the above remark we should be able to find a smaller genus surface with tight, filling multicurves whose intersection matrix is $Q$. We end this section by asking the following question:

\begin{question} For a given nonsingular positive integer matrix $Q$, what is the minimum genus surface that can be constructed, using the general construction, having tight filling multicurves $C$ and $D$ where $Q_{ij} = i(D_i, C_j)$?
\end{question}
	
\section{Initial Observations}

In this section we define Salem numbers, and look at two situations relating Salem numbers to Thurston's construction. The first observation will show that Salem numbers arise as stretch factors of pseudo-Anosov maps coming from Thurston's construction, and the second will lead us to a strategy for proving Theorem A.

\begin{definition}
A real algebraic unit $\lambda > 1$ is called a \textit{Salem number} if $\lambda^{-1}$ is a Galois conjugate, and all other conjugates lie on the unit circle.
\end{definition}

From this definition we can deduce the following properties of Salem numbers:

\begin{prop}
Let $\lambda$ be a Salem number. Then,
	\begin{enumerate}
			\item $\lambda^k$ is a Salem number for any positive integer $k$.\\
			\item $\lambda + \lambda^{-1}$ is a totally real algebraic integer.\\
			\item The Galois conjugates of $\lambda + \lambda^{-1}$ lie in the interval $(-2,2)$.
		\end{enumerate}
	\end{prop}

Since $\lambda + \lambda^{-1}$ is totally real, it is at least plausible that there are Salem numbers that are stretch factors arising from Thurston's construction. The following observation gives conditions when a Salem number arises as a stretch factor from Thurston's construction.

\vspace{1em}
\textbf{Observation 1.} \textit{There are Salem numbers that arise as stretch factors of pseudo-Anosov maps arising from Thurton's construction.}

\vspace{1em}
Suppose you have two tight, filling multicurves $C$ and $D$ where $|D| = 2$ and $|C| = k$, then $MQNQ^T$ is a $2 \times 2$ matrix. If $m_1 = m_2$, then it is also symmetric since $M$ commutes with $QNQ^T$. This matrix has two positive eigenvalues $\nu$ and $\mu$, where $\nu > \mu$ and which are roots of an integer polynomial
	\[x^2 - ax + b\]
where $\nu + \mu = a$ and $\nu\mu = b$. Now $T_C$ and $T_D$ are represented by the matrices in $(1)$ and their product gives
	\[[T_C][T_D] = \begin{bmatrix} 1- \nu & 1 \\ -\nu & 1\end{bmatrix}\]
The characteristic polynomial is $x^2 - (2 - \nu)x + 1$ and the polynomial 
\[f(x) = (x^2 - (2-\nu)x + 1)(x^2 - (2 -\mu)x + 1)\]
is an integer polynomial, and if we assume $|2 - \nu| > 2$ and $|2 - \mu| < 2$ then $f(x)$ will have two real roots, $\lambda$ and $\lambda^{-1}$, and two complex roots on the unit circle. This tells us that $\lambda$ is a Salem number of degree $4$ over the rationals.

\vspace{1em}
This situation is not difficult to set up.

\begin{example} Let $C = \lb C_1, C_2, C_3\rb$ and $D = \lb D_1, D_2\rb$ be the tight, filling multicurves on $S_2$ as in figure $1$.
We will choose $m_1 = m_2 = n_1 = n_2 = 2$ and $n_3 = 1$. Then
	\[
	   MQNQ^{T} = \begin{bmatrix} 8 & 4 \\ 4 & 6\end{bmatrix},
	\]
 where $\nu = 7 + \sqrt{17}$ and $\mu = 7 - \sqrt{17}$. We can see that $|2 - \nu| > 2$ and $|2 - \mu| < 2$, so $T_CT_D$ is a pseudo-Anosov map whose stretch factor is the Salem number 
\[\lambda = \df{5 + \sqrt{17} + \sqrt{38 + 10\sqrt{17}}}{2}.\]
\end{example}

The above is meant to highlight that Salem numbers arise without much effort from Thurston's construction, and they can appear when $MQNQ^T$ is symmetric. An obvious simplification then is if $Q$ is symmetric and  $M = N = I$.

\vspace{1em}
\textbf{Observation 2.} \textit{Salem numbers can arise as stretch factors when $Q$ is symmetric and $M = N = I$.}

\vspace{1em}
Suppose $S$ is a surface with tight, filling multicurves $C$ and $D$ with $|C| = |D|$ whose intersection matrix $Q$ (which is a square symmetric matrix) has $\lambda + \lambda^{-1}$ as an eigenvalue with positive eigenvector $\mathbf{v}$. Let $M = N = I$, then we have
	\[ 
		MQNQ^{T} = Q^2.
	\]
So we get that
\[
	[T_C] = \begin{bmatrix} 1 & 1\\ 0 & 1\end{bmatrix} \hspace{.2in}  \text{and} \hspace{.2in} [T_D] = \begin{bmatrix} 1 & 0 \\ -(\lambda + \lambda^{-1})^2 & 1\end{bmatrix}
\]
which gives
	\[[T_C][T_D] = \begin{bmatrix} 1 - (\lambda + \lambda^{-1})^2 & 1 \\ -(\lambda + \lambda^{-1})^2 & 1\end{bmatrix}\]
which has characteristic polynomial $x^2 + (\lambda^2 + \lambda^{-2})x + 1$ whose roots are $-\lambda^2$ and $-\lambda^{-2}$. Hence, up to projectivization, $T_CT_D$ has $\lambda^2$ as its stretch factor.

\vspace{1em}
With the above observations in mind our strategy for proving Theorem A is as follows: We will show that for any Salem number $\lambda$ there is a positive integer $k$ where $\lambda^k + \lambda^{-k}$ is the dominating eigenvalue of a positive symmetric integer matrix $Q$ with a corresponding positive eigenvector. From there we use the construction given in section 3 to build a closed orientable surface with two tight, filling multicurves having $Q$ as their intersection matrix. The discussion following observation $2$ shows that applying Thurston's construction to these multicurves with $M = N = I$ will give a pseudo-Anosov map with stretch factor $\lambda^{2k}$.

\section{Proof of Thereom A}

In this section we will assume that $\lambda$ is a Salem number. Some of the results here will be used to prove Theorem B, so we state them in their full generality with the goal of applying them to $\lambda + \lambda^{-1}$. Our starting point is the following theorem due to Estes, a proof of which can be found in \cite{Estes}.

\begin{thm}[Estes]
Let $\alpha$ be a totally real algebraic integer of degree $n$ over $\Q$, and $f(x)$ is its minimal polynomial. Then $\alpha$ is an eigenvalue of a rational symmetric matrix of size $(n + e) \times (n + e)$ whose characteristic polynomial is $f(x)(x-1)^e$, and $e \in \lb 1, 2\rb$.
\end{thm}

Since $\lambda + \lambda^{-1}$ is a totally real algebraic integer, then by Estes we know that there is a rational symmetric matrix $Q$ having the following eigenvalues: $\lambda + \lambda^{-1}$, the Galois conjugates of $\lambda + \lambda^{-1}$, and $1$ with multiplicity $1$ or $2$. Though Estes makes no claims about the eigenvectors related to these eigenvalues we can without loss of generality assume that $\lambda + \lambda^{-1}$ has a positive eigenvector. We do this by using the following:

\begin{prop}
The set $O(n;\Q)$, orthogonal matrices with rational entries, is a dense subgroup of $O(n)$. Consequently, $SO(n;\Q)$ is a dense subgroup of $SO(n)$.
\end{prop}

A proof for this can be found in \cite{schmutz}. We use Proposition $5.2$ as follows: Since every element $U$ of $SO(n)$ has the property that $U^{-1} = U^T$ then a symmetric matrix conjugated by an $SO(n)$ matrix remains symmetric. $Q$ has eigenvalue $\lambda + \lambda^{-1}$, so we conjugate $Q$ by an $SO(n)$ matrix so the resulting matrix has a positive eigenvector corresponding to $\lambda + \lambda^{-1}$. We then perturb the entries of $U$ so that they are all rational and now the resulting matrix is still rational and symmetric having a positive eigenvector $\mathbf{v}$ corresponding to $\lambda + \lambda^{-1}$.

\vspace{1em}
We now want a rational matrix having $\lambda$ as an eigenvalue with the goal of finding a power of that matrix that has integer entries. Consider the matrix
	\[\mathcal{M} = \begin{bmatrix} Q & -I \\ I & \mathbf{0}\end{bmatrix},\]

where $I$ is the $(n + e) \times (n + e)$ identity matrix. We will use powers of this matrix to find a positive symmetric integer matrix having $\lambda^k + \lambda^{-k}$ as an eigenvalue for some $k$. We will now establish several important properties:

\begin{prop} The characteristic polynomial of $\mathcal{M}$ is $p(x)(x^2 - x + 1)^e$ where $p(x)$ is the minimal polynomial of $\lambda$ over $\Q$. Therefore, $\mathcal{M}$ has integral characteristic polynomial, and $\det(\mathcal{M}) = 1$.
\end{prop}

\begin{proof} We will first show that $\mu$ is an eigenvalue for $\mathcal{M}$ if and only if $\mu + \mu^{-1}$ is an eigenvalue for $Q$. If $\mu$ is an eigenvalue for $\mathcal{M}$ then there is a vector $\begin{bmatrix} \mathbf{x} \\ \mathbf{y}\end{bmatrix}$ such that
	\begin{align*}
				\mathcal{M}\begin{bmatrix} \mathbf{x} \\ \mathbf{y}\end{bmatrix}   &= \begin{bmatrix} \mu\mathbf{x} \\ \mu\mathbf{y}\end{bmatrix}\\
				\begin{bmatrix}Q\mathbf{x} - \mathbf{y} \\ \mathbf{x}\end{bmatrix} &= \begin{bmatrix} \mu\mathbf{x} \\ \mu\mathbf{y}\end{bmatrix}
		\end{align*}
so $\mathbf{y} = \mu^{-1}\mathbf{x}$, and thus we have $Q\mathbf{x} = (\mu + \mu^{-1})\mathbf{x}$, therefore $\mu + \mu^{-1}$ must be an eigenvalue of $Q$.

\vspace{1em}
Now, if $\mu + \mu^{-1}$ is an eigenvalue of $Q$ with corresponding eigenvector $\mathbf{x}$, then we have
\begin{align*}
			\mathcal{M}\begin{bmatrix} \mathbf{x} \\ \mu^{-1}\mathbf{x}\end{bmatrix} &= \begin{bmatrix} Q\mathbf{x} - \mu^{-1}\mathbf{x} \\ \mathbf{x}\end{bmatrix}\\
																																									 &= \begin{bmatrix} (\mu + \mu^{-1})\mathbf{x} - \mu^{-1}\mathbf{x} \\ \mu\cdot\mu^{-1}\mathbf{x}\end{bmatrix}\\
																																									 &= \mu\begin{bmatrix} \mathbf{x} \\ \mu^{-1}\mathbf{x} \end{bmatrix}
		\end{align*}

Hence, $\mu$ is an eigenvalue of $\mathcal{M}$. Therefore, since $\lambda + \lambda^{-1}$ and its conjugates are eigenvalues of $Q$ then $\lambda$ and its conjugates are eigenvalues of $\mathcal{M}$. Also, since $1$ is an eigenvalue with multiplicity $e$ then $\mathcal{M}$ must have an eigenvalue $\mu$ such that $\mu + \mu^{-1} = 1$, in other words $\mu^2 - \mu + 1 = 0$, hence
	\[ \mu = \df{1 + i\sqrt{3}}{2} \hspace{.2in} \text{and} \hspace{.2in} \mu^{-1} = \df{1 - i\sqrt{3}}{2}\]
(where $\mu$ is a primitive $6$th root of unity) and if $\mathbf{y}$ is an eigenvector of $Q$ corresponding to $1$ then $\begin{bmatrix} \mathbf{y} \\ \mu^{-1}\mathbf{y}\end{bmatrix}$ and $\begin{bmatrix} \mathbf{y} \\ \mu\mathbf{y}\end{bmatrix}$ are eigenvectors of $\mathcal{M}$ corresponding to $\mu$ and $\mu^{-1}$, respectively. So if $1$ has multiplicity $e$, then $\mu$ and $\mu^{-1}$ are eigenvalues of $\mathcal{M}$ both with multiplicity $e$. Therefore, we have shown that the characteristic polynomial of $\mathcal{M}$ has the desired form.

\vspace{1em}
Finally, since $\lambda$ and $\mu$ are algebraic units then the characteristic polynomial of $\mathcal{M}$ is integral, and the product of the eigenvalues is $1$ so $\det(\mathcal{M}) = 1$.
\end{proof}

The fact that $\det(\mathcal{M}) = 1$ is true for any square matrix $Q$ since $I$ and $\mathbf{0}$ commute we have
	\[\det(\mathcal{M}) = \det(Q\cdot\mathbf{0} + I^2) = 1\]
but most importantly this tells us that $\mathcal{M}^{-1}$ exists. It is easy to check that 
	\[\mathcal{M}^{-1} = \begin{bmatrix} \mathbf{0} & I \\ -I & Q\end{bmatrix}\]
and notice that
	\[\mathcal{M} + \mathcal{M}^{-1} = \begin{bmatrix} Q & \mathbf{0} \\ \mathbf{0} & Q \end{bmatrix}\]

We will show that this behavior holds for all powers of $\mathcal{M}$.

\begin{prop}[Skew Property] $\mathcal{M}^k + \mathcal{M}^{-k}$ is a block diagonal matrix of the form $\begin{bmatrix} \mathcal{Q}_k & \mathbf{0} \\ \mathbf{0} & \mathcal{Q}_k\end{bmatrix}$ for any integer $k$. Here $\mathcal{Q}_k$ is a rational symmetric matrix whose characteristic polynomial is $g_k(x)(x - a)^e$, where $g_k(x)$ is the minimal polynomial of $\lambda^k + \lambda^{-k}$, and $a \in \lb-2,-1,1,2\rb$.
\end{prop}

\begin{proof} 
First we will show that for any $k$ we have $\mathcal{M}^k = \begin{bmatrix} Q_k & -Q_{k-1} \\ Q_{k-1} & Q_k - Q\cdot Q_{k-1}\end{bmatrix}$ where $Q_k$ is an integral combination of powers of $Q$. We define $Q_0 = I$ and $Q_1 = A$. We will proceed by induction: For $k = 2$ we get
	\[\mathcal{M}^2  = \begin{bmatrix} Q^2 - I & -Q \\ Q & -I\end{bmatrix}\]
So $Q_2 = Q^2 - I$, and $-I = Q_2 - Q\cdot Q_1$. Now assume that this form holds for $k$, then
	\[\mathcal{M}^{k} = \begin{bmatrix} Q_{k} & -Q_{k-1} \\ Q_{k-1} & Q_{k} - Q\cdot Q_{k-1}\end{bmatrix}\]
in which case we have
	\begin{align*}
			\mathcal{M}^{k+1} &= \begin{bmatrix} Q & -I \\ I & \mathbf{0}\end{bmatrix}\begin{bmatrix} Q_{k} & -Q_{k-1} \\ Q_{k-1} & Q_{k} - Q\cdot Q_{k-1}\end{bmatrix}\\
										    &= \begin{bmatrix} Q\cdot Q_{k} - Q_{k-1} & -Q_{k} \\ Q_{k} & -Q_{k-1} \end{bmatrix}
		\end{align*}
so if $Q_{k+1} = Q\cdot Q_k - Q_{k-1}$ then we have that
	\[\begin{bmatrix} Q_{k+1} & -Q_k \\ Q_k & Q_{k+1} - Q\cdot Q_k \end{bmatrix}\]

A similar inductive argument shows that $Q^{-k}$ has the form $\begin{bmatrix} Q_k - Q\cdot Q_{k-1} & Q_{k-1} \\ -Q_{k-1} & Q_k\end{bmatrix}$ for all $k$. Therefore, we have that
	\[\mathcal{M}^{k} + \mathcal{M}^{-k} = \begin{bmatrix} 2Q_k - Q\cdot Q_{k-1} & \mathbf{0} \\ \mathbf{0} & 2Q_k - Q\cdot Q_{k-1}\end{bmatrix}\]
so $\mathcal{Q}_k = 2Q_k - Q\cdot Q_{k-1}$, which is an integral combination of powers of $Q$ so $\mathcal{Q}_k$ is a rational symmetric matrix. Now that we have established the skew-property where $\mathcal{M}^{k} + \mathcal{M}^{-k}$ is a block diagonal matrix where the $(1,1)$-block and the $(2,2)$-block are equal rational symmetric matrices, then we can see that not only are the eigenvalues of $\mathcal{M}^{k} + \mathcal{M}^{-k}$ of the form $\nu_k + \nu_k^{-1}$, where $\nu_k$ is an eigenvalue of $\mathcal{M}^k$ but also that $\nu_k + \nu_k^{-1}$ is an eigenvalue of $\mathcal{M}^{k} + \mathcal{M}^{-k}$ if and only if $\nu_k + \nu_k^{-1}$ is an eigenvalue of each diagonal block.

\vspace{1em}
We immediately get from this that $\lambda^k + \lambda^{-k}$ and its conjugates are eigenvalues of $\mathcal{Q}_k$. Since $\mathcal{M}^k$ has $\mu^k$ and $\mu^{-k}$ as eigenvalues then we need to determine the possibilities for $\mu^k + \mu^{-k}$. Since $\mu$ is a primitive $6$th root of unity we have the following chart where $g_k(x)$ denotes the minimal polynomial of $\lambda^k + \lambda^{-k}$:
\begin{center}
	\begin{tabular}{ |l|c|c|c| } 
		\hline
$k \equiv b \mod 6$ & $\mu^k + \mu^{-k}$ & \text{Characteristic Polynomial of $\mathcal{Q}_k$} \\
		\hline
$b = 0$ & $2$ & $g_k(x)(x - 2)^e$\\
$b = 1,5$ & $1$ & $g_k(x)(x - 1)^e$\\ 
$b = 2,4$ & $-1$ & $g_k(x)(x + 1)^e$\\ 
$b = 3$   & $-2$ & $g_k(x)(x + 2)^e$\\ 
		\hline
	\end{tabular}
\end{center}
Therefore, we have proved Proposition $5.4$.
\end{proof}

Since $\mathcal{Q}_k$ is an integral combination of powers of $Q$, then it is clear that the eigenspaces of $\mathcal{Q}_k$ and $Q$ are exactly the same for all $k$. Therefore, since $Q$ has a positive eigenvector $\mathbf{v}$ corresponding to $\lambda + \lambda^{-1}$, then $\mathcal{Q}_k$ also has eigenvector $\mathbf{v}$ corresponding to $\lambda^k + \lambda^{-k}$. The goal now is to use this, in conjunction with the next proposition, to show that there is a $k$ where $\mathcal{Q}_k$ is a positive symmetric integral matrix.

\begin{prop}
Let $M \in M_n(\Q)$ such that $\det(M) = \pm 1$ and the characteristic polynomial of $M$ has integer coefficients. Then some power of $M$ is integral.
\end{prop}

\begin{proof} 
Since we can replace $M$ with $M^2$ then without loss of generality we can assume $\det(M) = 1$. Let $\Omega$ be the rational canonical form of $M$, which by assumption has integer entries. Thus, there is a nonsingular rational matrix $A$ such that $A^{-1}MA = \Omega$. Now $A = \frac{1}{c}P$ and $A^{-1} = \frac{1}{d}P'$ where
	\[PP' = cdI\]
$c,d$ are nonzero integers and $P,P'$ are integer matrices. Let $\overline{\Omega}$ be the matrix obtained from $\Omega$ by reducing its entries modulo $cd$. Since $\det(\Omega) = 1$ then $\det(\overline{\Omega}) = 1$, so  $\overline{\Omega} \in SL(n;\Z/cd\Z)$. Since $SL(n;\Z/cd\Z)$ is a finite group there is a positive integer $k$ such that
	\[\Omega^{\hspace{.01in}k} \equiv I \mod cd\]
that is, there an integer matrix $B$ such that $\Omega^k = I + cdB$. Now
\begin{align*}
		M^k &= A\Omega^kA^{-1}\\
				&= A(I + cdB)A^{-1}\\
				&= I + cdABA^{-1}\\
				&= I + PBP'
	\end{align*}
where $PBP'$ is an integer matrix, and thus $M^k$ is an integer matrix.
\end{proof}

Since $\mathcal{M}$ is a rational matrix with determinant $1$ and integer characteristic polynomial, then by the above proposition there is a $k$ such that $\mathcal{M}^k$ is an integer matrix. This means that $\mathcal{Q}_k$ is a symmetric integer matrix having $\lambda^k + \lambda^{-k}$ as an eigenvalue. We have now shown that every Salem number has a power $k$ such that $\lambda^k + \lambda^{-k}$ is an eigenvalue of a symmetric integral matrix. Now we want to show that we can raise $k$ high enough to get a positive matrix.

\vspace{1em}
$\mathcal{Q}_k$ is symmetric so we know there is an orthonormal basis of eigenvectors of $\mathcal{Q}_k$. Since there is a positive eigenvector $\mathbf{v}$ corresponding to $\lambda^{k} + \lambda^{-k}$ then for any standard basis vector $\mathbf{e}_i$
	\[\mathbf{e_i} = c_i\mathbf{v} + \mathbf{w}_i\]
where $\mathbf{w}_i$ lies in the orthogonal complement of Span$\lb\mathbf{v}\rb$, which is spanned by the other eigenvectors of $\mathcal{Q}_k$. Thus, for each $i$ we have $c_i = \mathbf{v}\cdot \mathbf{e_i} > 0$. Applying $\mathcal{Q}_k$ to both sides gives
\[\mathcal{Q}_k\mathbf{e}_i = c_i(\lambda^k + \lambda^{-k})\mathbf{e}_i + \mathcal{Q}_k\mathbf{w}_i\]
since $\lambda^k$ is a Salem number, the conjugates of $\lambda^k + \lambda^{-k}$ are real numbers in the interval $(-2,2)$. As we see above the vectors making up $\mathbf{w}_i$ are each scaled by a number in the interval $[-2,2]$, but $\lambda^k + \lambda^{-k}$ grows without bound, so eventually $\mathcal{Q}_k\mathbf{e}_i$ is a positive vector for any $i$.

\vspace{1em}
So, if we start with an integer $k$ such that $\mathcal{M}^k$ is integral, then $\mathcal{M}^{nk}$ is integral for any positive integer $n$, so choose $n$ large enough so that $\mathcal{Q}_{nk}$ is a positive symmetric integer matrix. Therefore, we have shown that for any Salem number $\lambda$ there is a positive integer $k$ such that $\lambda^k + \lambda^{-k}$ is an eigenvalue of a positive symmetric integral matrix.

\vspace{1em}
With this we can finish the proof as follows: Since $\mathcal{Q}_k$ is a nonsingular positive symmetric integer matrix having $\lambda^k + \lambda^{-k}$ as an eigenvalue with positive eigenvector $\mathbf{v}$, we can use Theorem $3.1$ to find a surface with tight, filling multicurves $C$ and $D$ whose intersection matrix is $\mathcal{Q}_k$. By Proposition $3.2$, the genus of this surface is $g = (n + e)^2 - (n + e) + 1$, where $n = [\Q(\lambda):\Q]$. Choosing $M = N = I$ and following discussion proceeding observation $2$, we see that $T_CT_D$ is a pseudo-Anosov map coming from Thurston's construction with stretch factor $\lambda^{2k}$.

\vspace{1em}
\textbf{Brief Summary.} Let $\lambda$ be a Salem number. Using Estes we find a rational symmetric matrix $Q$ having $\lambda + \lambda^{-1}$, its conjugates, and $1$ as eigenvalues. Without loss of generality we can assume that $Q$ has a positive eigenvector $\mathbf{v}$ corresponding to $\lambda + \lambda^{-1}$ since we can always conjugate $Q$ by an $SO(n;\Q)$ matrix to rotate an eigenvector of $\lambda + \lambda^{-1}$ into the first orthant. We define the matrix $\mathcal{M} = \begin{bmatrix} Q & -I\\ I & \mathbf{0}\end{bmatrix}$ which has $\lambda$, its conjugates, and the $6$th roots of unity $\mu$ and $\mu^{-1}$ as eigenvalues. Thus, $\mathcal{M}$ has determinant $1$, integer characteristic polynomial, $\mathcal{M}^{-1}$ exists and has the skew-property where $\mathcal{M}^k + \mathcal{M}^{-k}$ is a block diagonal matrix where the $(1,1)$ and $(2,2)$ blocks are equal; call these blocks $\mathcal{Q}_k$.

\vspace{1em}
$\mathcal{Q}_k$ has $\lambda^k + \lambda^{-k}$ as an eigenvalue, and all other eigenvalues lie in the interval $[-2,2]$ and by above arguments we can power up $\mathcal{M}$ so that $\mathcal{Q}_k$ is positive and integral. Using Theorem $3.1$, we know there is a surface with two tight filling multicurves $C$ and $D$ having $\mathcal{Q}_k$ as their intersection matrix. Applying Thurston's construction with $M = N = I$ we get a pseudo-Anosov map $T_CT_D$ having $\lambda^{2k}$ as its stretch factor. 

\section{Powers of Rational Symmetric Matrices}

Over the next two sections we will establish some final results, develop some background algebraic number theory and prove Theorem B. We start this section by proving a condition for when a real symmetric matrix will have a positive power and then conclude this section by showing that any rational symmetric matrix with a dominating eigenvalue larger than $1$ is conjugate to a rational symmetric matrix that has a positive power.

\begin{prop}
If $Q$ is a real symmetric matrix with a unique dominating eigenvalue $\lambda > 1$ and a positive eigenvector $\mathbf{v}$, then there is some power of $Q$ that is positive.
\end{prop}

\begin{proof}[Sketch of Proof.] The proof of this is very similar to the proof that $\mathcal{Q}_k$ is eventually positive. Every standard basis vector can be written as
\[\mathbf{e}_i = c_i\mathbf{v} + \mathbf{w}_i\]
 with $c_i > 0$ for each $i$, and applying $Q^k$ to both sides gives
\[Q^k\mathbf{e}_i = c_i\lambda^k\mathbf{v} + Q^k\mathbf{w}_i\]
Since $\lambda > \mu$ for all other eigenvalues $\mu$ then the sum $c_i\lambda^k\mathbf{v} + Q^k\mathbf{w}_i$ is eventually a positive vector. Hence, $Q^k\mathbf{e}_i$ is eventually positive for all $i$. So there is a $k$ such that $Q^k$ is a positive matrix.
\end{proof}

\begin{prop}
If $Q \in M_n(\Q)$ is symmetric, and has a unique dominating eigenvalue $\lambda > 1$ with corresponding eigenvector $\mathbf{v}$. Then there is a matrix $U \in SO(n;\Q)$ such that $UA^kU^{T}$ is a positive symmetric rational matrix for some positive integer $k$.
\end{prop}

\begin{proof} By Proposition $5.2$ $SO(n;\Q)$ is dense in $SO(n)$, so we can conjugate $Q$ by an $SO(n,\Q)$ matrix $U$ so that $\lambda$ now has a positive eigenvector. By Proposition $6.1$ we know there is some $k$ such that $UA^kU^{T}$ is a positive symmetric rational matrix.
\end{proof}

\section{Proof of Theorem B}

Before getting to the proof of Theorem B we will develop some algebraic number theory that we will use to prove the following:

\begin{thm} 
Given a totally real number field $K$, there is an algebraic unit $\alpha \in K$ such that $\alpha > 1$, $K = \Q(\alpha)$, $K = \Q(\alpha^m)$ for all positive integers $m$, and all conjugates of $\alpha$ are positive, less than 1.
\end{thm}

First a couple definitions:

\begin{definition} 
Given a basis $\mathbf{e}_1,...,\mathbf{e}_n$ for $\R^n$, then the $\mathbf{e}_i$ form a basis for a free $\Z$-module $L$ of rank $n$, namely,
	\[L = \Z\mathbf{e}_1 \oplus ... \oplus \Z\mathbf{e}_n\]
A set $L$ constructed this way is called a \textit{lattice} in $\R^n$.
\end{definition}

\vspace{1em} It is a well known fact that every number field $K$ of degree $n$ over $\Q$ has exactly $n$ embeddings into $\C$. Since we are talking about totally real number fields, these embeddings will be into $\R$. Let $\sigma_1,...,\sigma_n$ be these embeddings, where $\sigma_1$ denotes the inclusion embedding.

\begin{definition} Let $K$ be a totally real number field of degree $n$. We define the \textit{logarithmic embedding} of $K$ into $\R^n$ by
	\[ \lambda(x) = (\log|\sigma_1(x)|,...,\log|\sigma_n(x)|)\]
for all nonzero $x \in K$. Note: that $\lambda(xy) = \lambda(x) + \lambda(y)$, so $\lambda$ is a homomorphism from the multiplicative group $K^*$ to the additive group $\R^n$.
\end{definition}

The logarithmic embedding is used to prove the Dirichlet Unit Theorem, which gives a complete description of the unit group of a number field. A proof for this theorem can be found in \cite{MR2779252}, but we will just state it:

\begin{thm}[Dirichlet Unit Theorem]
Let $K$ be a number field, $r_1$ is the number of real embbedings, and $r_2$ is the number of complex embeddings (up to conjugacy). Then the unit group $U$ of $K$ is isomorphic to $G \times \Z^{r_1 + r_2 - 1}$, where $G$ is a finite cyclic group consisting of all the roots of unity in $K$.
\end{thm}

Since we are considering number fields $K$ that are totally real then if $[K: \Q] = n$ we have $r_1 = n$, $r_2 = 0$ and $G = \lb -1, 1\rb$. Hence, the unit group of $K$ is isomorphic to $\lb -1, 1 \rb \times \Z^{n-1}$. That is, there are units $u_1,...,u_{n-1}$ such that every unit of $K$ is of the form
 \[\pm u_{1}^{m_1}\cdots u_{n-1}^{m_{n-1}}\]
where $m_i$ are integers.

\vspace{1em}
The logarithmic embedding maps the unit group $U$ of $K$ to the hyperplane 
\[ H = \lb (x_1,...,x_n) \ \left\vert \ \displaystyle\sum_{i=1}^{n}{x_i} = 0\rb\right. .\]
In fact, if $u_1,...,u_{n-1}$ are the generators for $U$ then $\lb \lambda(u_1),...,\lambda(u_{n-1})\rb$ is a basis for $H$, hence $\lambda(U)$ is a lattice in $H$.

\begin{proof}[Proof of Theorem $7.1$.] \textit{Step} 1: We start by proving the following claim:

\vspace{1em}
\textit{Claim}: Suppose $\alpha$ is a unit that generates the field, which is bigger than 1, and all its Galois conjugates are less than $1$ in absolute value. Then $K = \Q(\alpha^m)$ for any positive integer $m$.

\begin{proof} Let $\lb \sigma_2(\alpha),...,\sigma_{n}(\alpha) \rb$ be the Galois conjugates of $\alpha$. If there is a positive integer $m$ such that $\Q(\alpha^m)$ is a proper subfield of $\Q(\alpha)$, then $\alpha^m$ has the following Galois conjugates (reordering if necessary) $\lb \sigma_2(\alpha)^m,...,\sigma_k(\alpha)^m\rb$, where $k < n$. Since $\alpha$ and $\alpha^m$ are algebraic units we have that
	\[	|\alpha\sigma_2(\alpha)\cdots\sigma_{n}(\alpha)| = 1\]
and
	\[		|\alpha^m\sigma_2(\alpha)^m\cdots\sigma_{k}(\alpha)^m| = 1.\]
This tells us that
\[   |\alpha\sigma_2(\alpha)\cdots\sigma_{k}(\alpha)| = 1\]
therefore
\[   |\sigma_{k+1}(\alpha)\cdots\sigma_{n}(\alpha)| = 1\]
which is impossible since $|\sigma_i(\alpha)| < 1$ for $i = 2,...,n$. Therefore, $K = \Q(\alpha^m)$ for any positive integer $m$.
\end{proof}

\textit{Step} 2: Now we will find such a unit. Let $u_1,...,u_{n-1}$ be positive generators for the unit group of $K$. Since $\lb \lambda(u_1),..., \lambda(u_{n-1})\rb$ is a basis for $H$ then we can take rational numbers $a_1,...,a_{n-1}$ such that
	\begin{align*}
				&(1) \ \displaystyle\sum_{i=1}^{n-1}a_i\log|u_i| > 1\\
				&(2) \ \displaystyle\sum_{i=1}^{n-1}a_i\log|\sigma_j(u_i)| < 0, \ 2\leq j \leq n - 1\\
				&(3) \ \text{The entries of } a_1\lambda(u_1)+...+a_{n-1}\lambda(u_{n-1}) \text{ sum to 0}\\
				&(4) \ \text{All entries of } a_1\lambda(u_1)+...+a_{n-1}\lambda(u_{n-1}) \text{ are distinct.}
		\end{align*}
Clearing denominators gives us an integer combination
	\[b_1\lambda(u_1)+...+b_{n-1}\lambda(u_{n-1})\]
that also has the above properties. Now take $u \in U$ to be the element
	\[ u = u_1^{b_1}\cdots u_{n-1}^{b_{n-1}}\]

By construction no two entries of $\lambda(u)$ are identical, hence $\sigma_i(u) \neq \sigma_j(u)$ for $i \neq j$, so $u$ has $n$ distinct Galois conjugates, therefore the minimal polynomial of $u$ over $\Q$ has degree $n$. Hence, $K = \Q(u)$. Also, by construction, $\log|u| > 1$ so $|u| > 1$ but $|\sigma_i(u)| < 1$ for all $2 \leq i \leq n$, so by step $1$ any power of $u$ generates $K$. Let $\alpha = u^2$, then $\alpha > 1$, all its conjugates are positive less than $1$, $K = \Q(\alpha)$ and $K = \Q(\alpha^m)$ for all positive integers $m$. 
\end{proof}

We now have all the pieces to prove the following:

\begin{lem} Let $K$ be a totally real number field. Then there is an algebraic unit $\alpha$ such that $K = \Q(\alpha)$ and $\alpha$ is the dominant eigenvalue of a positive symmetric integral matrix that is the intersection matrix of a pair of tight, filling multicurves on some orientable closed surface.
\end{lem}

\begin{proof} 
Let $K$ be a totally real number field, then by Theorem $7.1$ we can find an algebraic unit $\beta$ so that $\beta > 1$, $K=\Q(\beta^m)$ for all positive integers $m$, and all conjugates of $\beta$ are positive less than $1$. By Theorem $5.1$ (Estes) there is a rational symmetric matrix $B$ having $\beta$ as its unique dominating eigenvalue greater than 1, and whose characteristic polynomial is $f(x)(x - 1)^e$, where $f(x)$ is the minimal polynomial of $\beta$ over $\Q$. Note that $f(x)$ has integer coefficients and since $\beta$ is a unit then the constant term of $f(x)$ is 1, so $\det(B) = \pm1$.

\vspace{1em}
Now, by Proposition $6.2$ we can conjugate $B$ by an $SO(n + e; \Q)$ matrix $U$ where there is a positive integer $k$ such that $Q = UB^kU^{T}$ is a positive rational matrix, without loss of generality assume $k$ is even. Now, $Q$ is a positive rational matrix whose eigenvalues are $\beta^k$, its conjugates, and $1$, hence the coefficients of the characteristic polynomial of $Q$ are integers, and $\det(Q) = (\pm1)^k = 1$.

\vspace{1em}
Applying Proposition $5.5$, we know there is some positive integer $\ell$ so that $Q^{\ell}$ is integral and if we let $\alpha = \beta^{k\ell}$ then $ K = \Q(\alpha)$ where $\alpha$ is the dominating eigenvalue of $Q^{\ell}$, which is a positive symmetric integral matrix. Now apply Theorem $3.1$ to find a closed orientable surface with multicurves $C = \lb C_1,...,C_{(n + e)}\rb$ and $D = \lb D_1,...,D_{(n+e)}\rb$ such that $i(D_i,C_j) = Q_{ij}^{\ell}$. 
\end{proof}

We end the paper by proving Theorem B, which we restate here.

\vspace{1em}
\textbf{Theorem B.} \textit{Every totally real number field is of the form $K = \Q(\lambda + \lambda^{-1})$, where $\lambda$ is the stretch factor of a pseudo-Anosov map arising from Thurston's construction.}

\begin{proof} Let $K$ be a totally real number field. By Lemma $7.5$ we can find a unit $\alpha$ such that $K = \Q(\alpha)$ and $\alpha$ is the dominating eigenvalue of a positive symmetric integral matrix $Q$ that is the intersection matrix of a pair of tight, filling multicurves, $C$ and $D$, on a closed orientable surface $S$. Without loss of generality we can assume $\alpha > 2$. Applying Thurston's construction with $M = N = I$ we have $\alpha^2$ as the dominating eigenvalue of $MQNQ^T = Q^2$, and the following representations of $T_C$ and $T_D$:
	\[
				[T_C] = \begin{bmatrix} 1 & 1 \\ 0 & 1\end{bmatrix} \hspace{.1in} \text{and} \hspace{.1in} [T_D] = \begin{bmatrix} 1 & 0 \\ -\alpha^2 & 1\end{bmatrix}
	\]
Multiplying these matrices gives
	\[
			[T_C][T_D] = \begin{bmatrix} 1 - \alpha^2 & 1 \\ -\alpha^2 & 1\end{bmatrix}
	\]
Since $\alpha^2 > 4$ then we have that $|tr([T_C][T_D])| = |2 - \alpha^2| > 2$. So $T_CT_D$ is a pseudo-Anosov map with stretch factor
	\[
			\lambda = \df{(\alpha^2 - 2) + \alpha\sqrt{\alpha^2 - 4}}{2}
	\]
Hence, $\lambda + \lambda^{-1} = \alpha^2 - 2$, and so $\Q(\lambda + \lambda^{-1}) = \Q(\alpha^2 - 2) = \Q(\alpha^2) = K$. 
\end{proof}
\bibliographystyle{amsalpha}
\bibliography{mybib}{}

\providecommand{\bysame}{\leavevmode\hbox to3em{\hrulefill}\thinspace}
\providecommand{\MR}{\relax\ifhmode\unskip\space\fi MR }
\providecommand{\MRhref}[2]{%
  \href{http://www.ams.org/mathscinet-getitem?mr=#1}{#2}
}
\providecommand{\href}[2]{#2}
\begin{thebibliography}{FLP12}

\bibitem[Ash10]{MR2779252}
Robert~B. Ash, \emph{A course in algebraic number theory}, Dover Publications,
  Inc., Mineola, NY, 2010. \MR{2779252}

\bibitem[Est92]{Estes}
Dennis~R. Estes, \emph{Eigenvalues of symmetric integer matrices}, J. Number
  Theory \textbf{42} (1992), no.~3, 292--296. \MR{1189507}

\bibitem[FLP12]{FLP}
Albert Fathi, Fran\c{c}ois Laudenbach, and Valentin Po\'enaru, \emph{Thurston's
  work on surfaces}, Mathematical Notes, vol.~48, Princeton University Press,
  Princeton, NJ, 2012, Translated from the 1979 French original by Djun M. Kim
  and Dan Margalit. \MR{3053012}

\bibitem[FM12]{MR2850125}
Benson Farb and Dan Margalit, \emph{A primer on mapping class groups},
  Princeton Mathematical Series, vol.~49, Princeton University Press,
  Princeton, NJ, 2012. \MR{2850125}

\bibitem[Lon85]{Long}
D.~D. Long, \emph{Constructing pseudo-{A}nosov maps}, Knot theory and manifolds
  ({V}ancouver, {B}.{C}., 1983), Lecture Notes in Math., vol. 1144, Springer,
  Berlin, 1985, pp.~108--114. \MR{823284}

\bibitem[Pen88]{MR930079}
Robert~C. Penner, \emph{A construction of pseudo-{A}nosov homeomorphisms},
  Trans. Amer. Math. Soc. \textbf{310} (1988), no.~1, 179--197. \MR{930079}

\bibitem[Sch08]{schmutz}
Eric Schmutz, \emph{Rational points on the unit sphere}, Cent. Eur. J. Math.
  \textbf{6} (2008), no.~3, 482--487. \MR{2425007}

\bibitem[SS15]{MR3447112}
Hyunshik Shin and Bal\'azs Strenner, \emph{Pseudo-{A}nosov mapping classes not
  arising from {P}enner's construction}, Geom. Topol. \textbf{19} (2015),
  no.~6, 3645--3656. \MR{3447112}

\bibitem[Thu88]{MR956596}
William~P. Thurston, \emph{On the geometry and dynamics of diffeomorphisms of
  surfaces}, Bull. Amer. Math. Soc. (N.S.) \textbf{19} (1988), no.~2, 417--431.
  \MR{956596}

\end{thebibliography}
\end{document}